\documentclass[a4paper,UKenglish,cleveref, autoref, thm-restate,]{lipics-v2021}
\nolinenumbers
\hideLIPIcs

\Crefname{figure}{Figure}{Figures}

\crefformat{equation}{\textup{#2(#1)#3}}
\crefrangeformat{equation}{\textup{#3(#1)#4--#5(#2)#6}}
\crefmultiformat{equation}{\textup{#2(#1)#3}}{ and \textup{#2(#1)#3}}
{, \textup{#2(#1)#3}}{, and \textup{#2(#1)#3}}
\crefrangemultiformat{equation}{\textup{#3(#1)#4--#5(#2)#6}}%
{ and \textup{#3(#1)#4--#5(#2)#6}}{, \textup{#3(#1)#4--#5(#2)#6}}{, and \textup{#3(#1)#4--#5(#2)#6}}

\Crefformat{equation}{#2Equation~\textup{(#1)}#3}
\Crefrangeformat{equation}{Equations~\textup{#3(#1)#4--#5(#2)#6}}
\Crefmultiformat{equation}{Equations~\textup{#2(#1)#3}}{ and \textup{#2(#1)#3}}
{, \textup{#2(#1)#3}}{, and \textup{#2(#1)#3}}
\Crefrangemultiformat{equation}{Equations~\textup{#3(#1)#4--#5(#2)#6}}%
{ and \textup{#3(#1)#4--#5(#2)#6}}{, \textup{#3(#1)#4--#5(#2)#6}}{, and \textup{#3(#1)#4--#5(#2)#6}}

\usepackage{hyperref}
\usepackage{graphicx} % Required for inserting images

\usepackage{amsmath,amssymb,amsthm}
\usepackage{bm}
\usepackage{todonotes}
\usepackage{mathtools,xcolor}

\newcommand{\R}{\mathbb{R}}
\newcommand{\Z}{\mathbb{Z}}
\newcommand{\Q}{\mathbb{Q}}
\newcommand{\N}{\mathbb{N}}
\newcommand{\puiseux}{\mathbb{K}}
\DeclareMathOperator*{\val}{\mathsf{val}}
\DeclareMathOperator*{\sign}{\mathrm{sign}}

\DeclareMathOperator*{\bigland}{\bigwedge}
\DeclareMathOperator*{\biglor}{\bigvee}

\newcommand{\dgraph}{\vec{\mathcal{G}}}
\newcommand{\vertices}{V}
\newcommand{\edges}{E}
\newcommand{\Max}{\mathrm{Max}}
\newcommand{\Min}{\mathrm{Min}}
\newcommand{\Rand}{\mathrm{Rand}}
\newcommand{\Maxvertices}{\vertices_{\Max}}
\newcommand{\Minvertices}{\vertices_{\Min}}
\newcommand{\Randvertices}{\vertices_{\Rand}}
\newcommand{\Win}{\mathrm{Win}}
\newcommand{\Lose}{\mathrm{Lose}}
\newcommand{\payoff}{r}

\title{Reducing Stochastic Games to Semidefinite Program Feasibility}

\author{Manuel Bodirsky}{Institut f\"ur Algebra, TU Dresden, Germany}{manuel.bodirsky@tu-dresden.de}{https://orcid.org/0000-0001-8228-3611}{The author received funding from the ERC (Grant Agreement no. 101071674, POCOCOP). Views and opinions expressed are however those of the authors only and do not necessarily reflect those of the European Union or the European Research Council Executive Agency.}
\author{Georg Loho}{FU Berlin, Germany \and U Twente, The Netherlands}{georg.loho@math.fu-berlin.de}{https://orcid.org/0000-0001-6500-385X}{}
\author{Mateusz Skomra}{LAAS-CNRS, Université de Toulouse, CNRS, Toulouse, France}{mateusz.skomra@laas.fr}{https://orcid.org/0000-0001-5650-5559}{The author has been supported by European Union’s HORIZON–MSCA-2023-DN-JD programme under the Horizon Europe (HORIZON) Marie Skłodowska-Curie Actions, grant agreement 101120296 (TENORS).}

\authorrunning{M. Bodirsky, G. Loho, and M. Skomra}

\funding{This research benefited from the support of the FMJH Program PGMO.}

\relatedversion{}
\relatedversiondetails{An extended abstract with a weaker result, namely an SDP formulation for max-average constraints rather than max-plus-average constraints, was presented at ICALP 2025~\cite{SG-SDP}. Full Version}{https://arxiv.org/abs/2411.09646}



\Copyright{Manuel Bodirsky, Georg Loho, and Mateusz Skomra}

\begin{document}

\maketitle

\begin{abstract}
    We present a polynomial-time reduction from max-plus-average constraints to the feasibility problem for semidefinite programs. This shows that Condon's simple stochastic games, stochastic mean payoff games, and in particular mean payoff games and parity games can all be reduced to semidefinite programming. 
\end{abstract}

\section{Introduction}
There are two important clusters of computational problems that are not known to be in the complexity class P: the first is related to numeric computation, and contains the sums-of-square-roots problem~\cite{GGJ76,Eisenbrand2024}, the Euclidean shortest path problem, PosSLP~\cite{ABKM08}, and the feasibility problem for semidefinite programs~\cite{nesterovnemirovskibook,Ramana}. The second cluster is related to tropical geometry, and contains for instance the model checking problem for the propositional $\mu$-calculus~\cite{EmersonJutla,emerson_jutla_sistla}, parity games~\cite{MartinBorel,MostkowskiGames}, mean payoff games~\cite{EhrenfeuchtMycielski}, and/or scheduling~\cite{and-or-scheduling}, stochastic mean payoff games~\cite{andersson_miltersen}, simple stochastic games~\cite{condon}, tropically convex and 
max-plus-average constraints~\cite{BodirskyMaminoTropical}. 
So far, no polynomial-time reduction from a problem of one of the clusters to a problem from the other cluster was known. We show that all of the mentioned problems from the second cluster 
can be reduced in polynomial time 
to the feasibility problem for semidefinite programs from the first cluster.

Semidefinite programming is a generalisation of linear programming with many important algorithmic applications in discrete and continuous optimisation, both from a theoretical and a practical perspective~\cite{GroetschelLovaszSchrijver,ben-tal_nemirovski,lasserre_book2009}.
However, already the feasibility problem for semidefinite programs is not known to be in P. 
By Ramana's duality~\cite{Ramana}, the problem is either in the intersection of NP and coNP, or outside of the union of NP and coNP. 
The semidefinite feasibility problem falls into the existential theory of the reals, which is known to be in PSPACE~\cite{Canny:1988:AGC:62212.62257}. 
However, it has been observed by Khachiyan that the smallest feasible solution to an SDP might be of doubly exponential size (see, e.g.,~\cite{KhachiyanPorkolab}), which is an obstacle for the polynomial-time algorithmic methods known for linear programming. In fact, it is possible to reduce the \emph{PosSLP problem} (testing positivity of a value computed by a \emph{Straight Line Program}) to semidefinite programming~\cite{Tarasov20082070}. 
PosSLP has been introduced in~\cite{ABKM08}, motivated by the `generic task of numerical computation'. For example, the sums-of-square-roots problem, which is a numeric computational problem not even known to be in NP and the barrier for polynomial-time tractability for numerous problems in computational geometry, such as the Euclidean shortest path problem, has a polynomial-time reduction to PosSLP~\cite{ABKM08}. 

\emph{Mean payoff games} are turn-based two-player games on graphs~\cite{EhrenfeuchtMycielski,ZwickPaterson}; it is a famous open problem in theoretical computer science whether there is a polynomial-time algorithm to decide for a given graph which of the players has a winning strategy. Finding the winning region in mean payoff games is polynomial-time equivalent to various other computational problems in theoretical computer science, for instance 
scheduling with and-or-constraints~\cite{and-or-scheduling}, the max-atoms problem~\cite{Max-atoms}, as well as solvability of max-plus systems and tropical linear feasibility~\cite{MaxPlus,polyhedra_equiv_mean_payoff}. 
The latter can be seen as the tropical analog of testing feasibility of linear inequalities (which for classical geometry is known to be in P, e.g., via the ellipsoid method~\cite{Khachiyan,GroetschelLovaszSchrijver}). 

Furthermore, there is a simple reduction from \emph{parity games} to mean payoff games by using the priorities of the parity game with a suitably large basis \cite{puri1995theory}. 
Parity games are polynomial-time equivalent to the model-checking problem of the propositional $\mu$-calculus~\cite{emerson_jutla_sistla}, which has been called \emph{`the most important logics in model checking'}~\cite{BradfieldIgor}. 
For parity games, a quasipolynomial algorithm has been found recently~\cite{CaludeJKLS22}. 
Subsequent work indicates that the various quasipolynomial methods that have been obtained lately for parity games~\cite{LehtinenPSW22,JurdzinskiMT22,LehtinenB20} 
do not extend to mean payoff games~\cite{ColcombetFGO22}.

Mean payoff games can be generalised to \emph{stochastic mean payoff games} (sometimes called \emph{$2 \frac{1}{2}$-player games}), where the graph might contain stochastic nodes. If the strategy of one of the two players in such a game is fixed, the game turns into a Markov decision process, for which polynomial-time algorithms based on linear programming are known~\cite{puterman}. Stochastic mean payoff games 
are equivalent under polynomial-time Turing reductions to Condon's \emph{simple stochastic games}~\cite{andersson_miltersen} (we may even assume that the simple stochastic games are \emph{stopping}, see the discussion in \cref{sect:SG}), which are known to be in the intersection of NP and coNP~\cite{condon}. There are many other variants of games that all have the same complexity, see~\cite{andersson_miltersen}. 

One approach to analyze mean payoff games consists in reducing these games to \emph{constraint satisfaction problems (CSPs)}. More precisely, it is known that mean payoff games are equivalent to the max-atom problem~\cite{atsma,and-or-scheduling}, while {the more expressive} stochastic mean payoff games can be reduced to max-plus-average constraints, 
which are still in NP $\cap$ coNP~\cite{BodirskyMaminoTropical}. In the conference version of this article, we introduced a fragment of max-plus-average constraints, called \emph{max-average constraints}. 
There is a polynomial-time reduction from stopping simple stochastic games to max-average constraints (\cref{sect:max-avg}). 
Max-average constraints have the advantage that they can be further reduced to the 
feasibility problem for semidefinite programs, as we will show here. 
In fact, with some extra effort we can also show this for the entire class of max-plus-average constraints. 
This implies that all the mentioned computational problems from the second cluster can be reduced to semidefinite programming. 
See \cref{fig:landscape} for an overview of the mentioned computational problems and their relationship.

\begin{figure}
\begin{center}
\includegraphics[scale=.5]{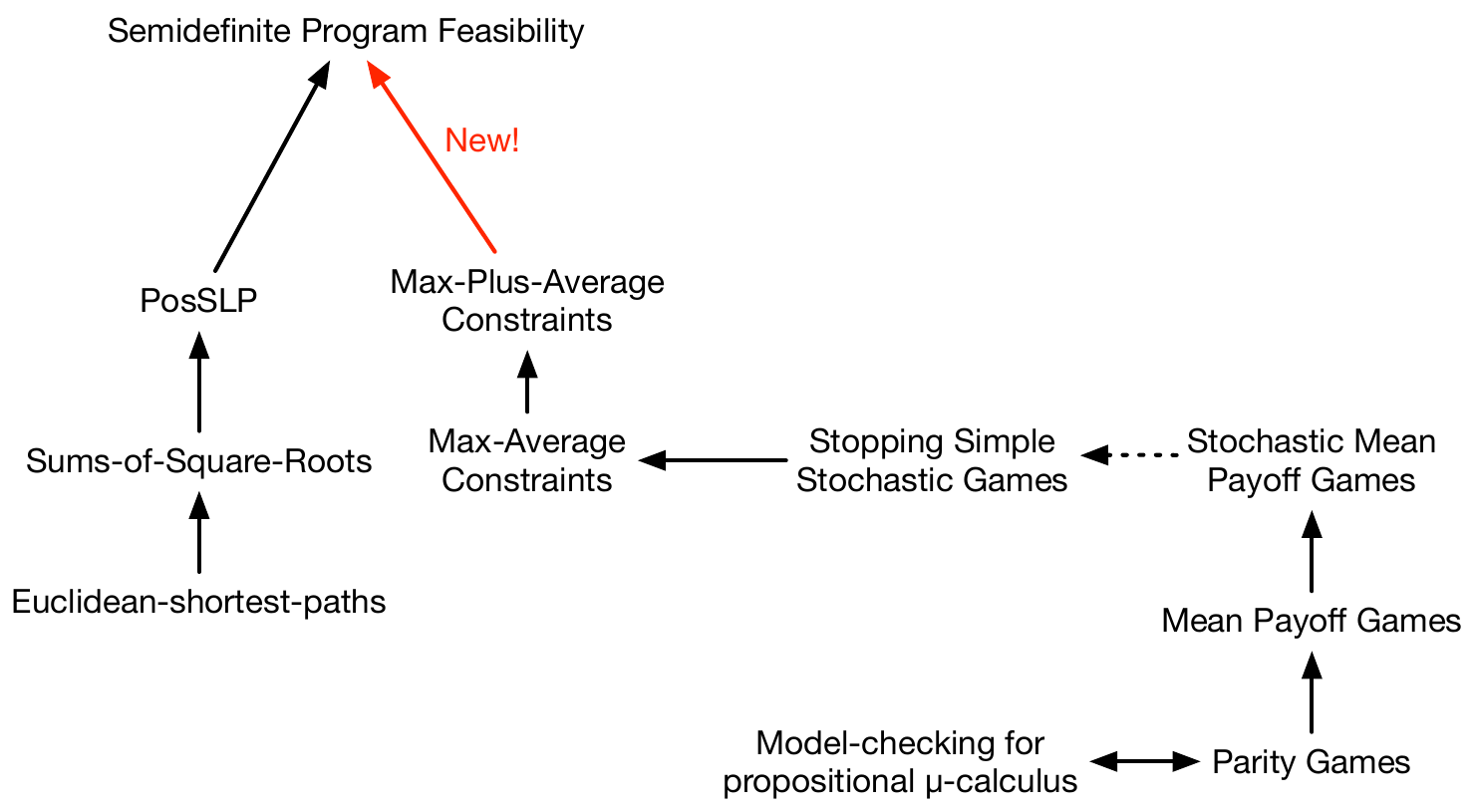}
    \end{center}
    \caption{Computational problems that are not known to be in P and not known to be NP-hard, and their relationship. Arcs indicate polynomial-time (many-one) reductions, the dotted arc indicates a polynomial-time Turing reduction.}
    \label{fig:landscape} 
\end{figure}

Our reduction consists of two steps. The first step uses a similar idea as the reduction of Schewe~\cite{schewe2009parity} 
to reduce parity games and mean payoff games to linear programming, which is, however, not a proper polynomial-time reduction. More precisely, the idea of this first step is to replace constraints of the form $x_0 = \max(x_1,\dots,x_k)$, which occur in the max-atom problem, by
$x_0 = \log_b(b^{x_1} + \dots + b^{x_k})$
for a sufficiently large $b$. This reduction is not a polynomial-time reduction since the numbers $b^{x_i}$ may have exponential bit-sizes for any $x$ that solves the original problem. 
Following \cite{issac2016jsc}, we extend this idea to max-plus-average constraints and obtain a proper polynomial-time reduction to non-Archimedean SDPs, which are SDPs defined over a field of formal power series. 
From there, we use bounds from quantifier elimination results to translate non-Archimedean SDPs to real SDPs. However, we still have the obstacle that these SDPs involve coefficients of doubly exponential size, which means that they have exponential representation size and do not lead to a polynomial-time reduction. In the second step of our reduction, we overcome this difficulty by using the duality of semidefinite programming to find small expressions that define coefficients of doubly exponential size, which combined with the above ideas leads to a polynomial-time reduction from max-plus-average constraints to (real) SDPs. 

Our results imply that if there is a  polynomial-time algorithm for SDPs,
then all the mentioned problems in the second cluster can be solved in polynomial time as well. 
Conversely, our reduction can be used to translate interesting families of instances for (stochastic) mean payoff games or parity games into SDP instances, which might yield interesting families of instances for (algorithmic approaches to) semidefinite programming. 
For linear programming, this idea turned out to be very fruitful: 
certain instances of parity games have been used to prove exponential lower bounds for several famous pivoting rules for the simplex method in linear programming~\cite{Friedmann2011b,friedmann_zadeh,FriedmannHZ14a,AvisF17,DisserFriedmannHopp:2023}.
However, these results were not based on a general polynomial-time reduction from parity games to linear programming (which is not known!). We also hope that algorithmic ideas for solving games might generalise to (certain classes of) SDPs. 

Our results also show that two large research communities (convex real algebraic geometry, and more generally convex optimisation on the one hand, and verification and automata theory on the other hand) were working on closely related computational problems and might profit from joining forces in the future.

\subsection{Related work}
There are multiple works that reduce 
games to continuous optimization problems. As mentioned above, one-player mean payoff games (Markov decision processes) can be reduced to linear programming \cite{puterman}, which implies that they can be solved in polynomial time (but it is open whether there is a strongly polynomial algorithm).
For the case of two-player deterministic mean payoff games, Schewe~\cite{schewe2009parity} proposed a reduction to linear programs with coefficients of exponential representation size. 
Such programs can also be interpreted as linear programs over non-Archimedean ordered fields \cite{tropical_simplex,combinatorial_mean_payoff,AllamigeonBenchimolGaubertJoswig:2021}.
The idea of Schewe was generalized to stochastic games in \cite{BorosElbassioniGurvichMakino:2017}, where the authors prove that stochastic mean payoff games can be encoded by convex programs with constraints of exponential encoding size. 
Such programs can be solved by the ellipsoid method in pseudo-polynomial time provided that the number of stochastic nodes is fixed. 
The non-Archimedean approach was generalized from deterministic games to stochastic games in the work \cite{issac2016jsc}, which proves that stochastic mean payoff games can be reduced to semidefinite programs over non-Archimedean fields. 
The paper \cite{issac2016jsc} serves as a basis for our results. 

These reductions express mean payoff games as convex optimization problems, but the resulting problems are ``non-standard'' in the sense that they either use constraints of exponential encoding size or are expressed over non-Archimedean fields. 
There is yet another line of work deriving polynomial-time reductions from mean payoff games to classical optimization problems which are not convex. 
In particular, \cite{Condon:1993} expressed simple stochastic games as non-convex quadratic programs. 
This reduction is further studied in \cite{Kretinsky2022comparison}. 
Another reduction to (generalized) complementarity problems with $P$-matrices was proposed in \cite{GaertnerRuest:2005,SvenssonVorobyov:2006} and extended in subsequent works \cite{JurdzinskiSavani:2008,FearnyJurdzinskiSavani:2010,HansenIbsenJensen:2013}. 
Such linear complementarity problems form a subclass of quadratic programs and belong to the complexity class UniqueEOPL~\cite{fearnley2020unique}. 
However, they are in general non-convex and can therefore not be expressed by semidefinite programs.

\section{Preliminaries}
The set of natural numbers 
 $\{0,1,2,\dots\}$ is denoted by 
${\mathbb N}$, and the set of real numbers by ${\mathbb R}$. 
We assume familiarity with first-order logic; see, e.g.,~\cite{Hodges}. 
Let $\tau$ be a \emph{signature}, i.e., a set of function and relation symbols. A first-order $\tau$-formula
is a formula built from the Boolean connectives, atomic $\tau$-formulas and quantification in the usual way.
If $A$ is a $\tau$-structure and $\phi$ is a $\tau$-formula with free variables $x_1,\dots,x_k$, then 
$\phi$ \emph{defines} over $A$ the relation
\[
\{(a_1,\dots,a_k) \in A^k \mid A \models \phi(a_1,\dots,a_k)\}.
\]
A set is called \emph{semialgebraic} if it has a first-order definition over the field $({\mathbb R};+,\cdot)$, allowing real parameters (we do not need the order in the signature since it is definable from addition and multiplication). A first-order formula $\phi(x_1,\dots,x_k)$ with free variables $x_1,\dots,x_k$ is called \emph{primitive positive} if it is of the form 
\[
\exists y_1,\dots,y_{\ell} (\psi_1 \wedge \cdots \wedge \psi_m)
\]
where $\psi_1,\dots,\psi_m$ are atomic formulas over the variables $x_1,\dots,x_k,y_1,\dots,y_{\ell}$. 
A relation is called \emph{primitive positive definable in a $\tau$-structure $A$} if there exists a primitive positive $\tau$-formula which defines the relation over $A$.

\subsection{Semidefinite Programming}
For a matrix $S \in {\mathbb R}^{m \times m}$ we write 
\begin{itemize}
    \item $S \succeq 0$ if $S$ is \emph{positive semidefinite}, i.e., $S$ is symmetric and 
$x^{\top} S x \geq 0$ for all $x \in {\mathbb R}^m$,
and 
\item $S \succ 0$ if it is \emph{positive definite}, i.e., $S$ is symmetric and  
$x^{\top} S x > 0$ for all $x \in {\mathbb R}^m \setminus \{0\}$. 
\end{itemize}
A \emph{semidefinite program (SDP)} is an optimisation problem of the following form:
\begin{equation}\label{eq:P}\tag{P}
\begin{aligned}
    \inf \ &\sum_{i=1}^n c_i x_i \nonumber \\ 
     \text{such that} \ &\sum_{i=1}^n x_i A_i - B \succeq 0 \, , 
\end{aligned}
\end{equation}
where $A_1,\dots,A_n,B \in {\mathbb R}^{m \times m}$ are symmetric and $c_1,\dots,c_n \in {\mathbb R}$. 
If the data $A_1,\dots,A_n,B$, $c_1,\dots,c_n$ is rational, then the problem of deciding whether the infimum exists, and if so, to compute the infimum, is a computational problem known as the \emph{semidefinite programming problem}. The \emph{representation size} of the SDP then consists of the sum of the bit lengths of all the rational numbers in the input. 

The expression $\sum_{i=1}^n x_i A_i -B \succeq 0$ in~\cref{eq:P} is called a \emph{linear matrix inequality (LMI)}, and defines the \emph{feasible region} of the SDP, which  
is called a  \emph{spectrahedron}. 
The representation size of LMIs is defined analogously as for SDPs. 
We say that an SDP is \emph{strictly feasible} if $\sum_{i=1}^n x_i A_i - B \succ 0$ has a solution $x \in {\mathbb R}^n$.

\begin{definition}\label{def:sdfp}
    Given symmetric matrices $A_1,\dots,A_n,B \in {\mathbb Q}^{m \times m}$ the \emph{feasibility problem for semidefinite programs} asks to decide whether there is a solution $x \in \R^n$ to an LMI  
    \[
    x_1A_1 + \dots + x_nA_n - B \succeq 0 \enspace .
    \]
\end{definition}

Note that spectrahedra are convex and semialgebraic.
It has been conjectured that every convex semialgebraic set is the projection of a spectrahedron (also called a \emph{spectrahedral shadow}), but this turned out to be false~\cite{scheiderer_helton_nie} (see~\cite{BodKummerThom} for an alternative proof of an even stronger inexpressibility result). 
It is easy to see that the set of spectrahedral shadows is closed under primitive positive definability, and that 
for every primitive positive formula $\phi$ 
we may find an LMI for the 
spectrahedral shadow defined by $\phi$ whose representation size is bounded by the
sum of the representation sizes of conjuncts of $\phi$. We also note that the definition of a semidefinite program makes sense over any real closed field, not only over the real numbers. In particular, it is meaningful to consider semidefinite programs over non-Archimedean real closed fields, such as the field of Puiseux series which we discuss in \cref{sec:nonArchSDP}.

The \emph{(Frobenius) inner product} of two matrices $S,T \in {\mathbb R}^{m \times m}$ is defined as $\langle S,T \rangle := \text{trace}(ST) = \sum_{i,j} S_{ij} T_{ij}$. 
The \emph{dual program} to the SDP from~\cref{eq:P} (to which we refer as the \emph{primal}) is given by
\begin{equation}\label{eq:dual}\tag{D}
\begin{aligned}
    \sup_{Y \succeq 0} \ &\langle B, Y \rangle %
    \\
     \text{such that}  \ &\langle A_i,Y \rangle = c_i \text{ for every } i \in \{1,\dots,n\} 
     \, .
\end{aligned}
\end{equation}

It is well-known that the value of the primal SDP is bounded from below by the value of the dual; however, the lower bound might not be tight, i.e., there might be a non-zero \emph{duality gap}. Moreover, 
both primal and dual might not have optimal solutions, i.e., the supremum may not be a maximum and the infimum might not be a minimum. See \cite{Ramana} for a more complicated, but exact dual. However, for us the following sufficient condition for a vanishing duality gap will be enough. 

\begin{theorem}[see Theorem 2.4.1 in~\cite{ben-tal_nemirovski} and the explanations on page 142 and 143]\label{thm:strong-dual}
    If the primal SDP is strictly feasible and bounded from below, then the dual program attains its supremum and there is no duality gap.
\end{theorem}

\subsection{Stochastic Games}
\label{sect:SG}
A \emph{stochastic mean payoff game} is a two-player game played on a finite edge-weighted directed graph $\dgraph = (V,E,\payoff,p)$,
i.e., $V$ is a finite set of nodes, $E \subseteq V^2$ is a set of edges, $r \colon E \to {\mathbb Q}$ is a \emph{weight function},
and $p \colon E \to {\mathbb Q} \cap [0,1]$
specifies \emph{edge probabilities}. 
The nodes of the graph are partitioned into three subsets, $V = \Minvertices \sqcup \Maxvertices \sqcup \Randvertices$. The nodes in $\Minvertices$ are controlled by player Min, the nodes in $\Maxvertices$ are controlled by player Max, and the nodes in $\Randvertices$ are controlled by nature; the probabilities of the outgoing edges of each vertex in $\Randvertices$ sum up to one. 
The players play the game by moving a pawn on the graph. When the pawn is on a node $v \in \Minvertices$, player Min chooses an outgoing edge $(v,u)$ and moves the pawn to $u$. Analogously, player Max chooses the next move when the pawn is on a node $v \in \Maxvertices$. When the pawn is on a node $v \in \Randvertices$, the next move is chosen by nature randomly according to the probability distribution given by $p$. 
To ensure that players can always make a move, we suppose that all nodes of the graph have at least one outgoing edge. The weights of the edges of the graph represent payoffs that player Max receives from player Min after each move of the pawn. 
In order to state the objectives of both players, we first specify the strategies that they are allowed to use. A \emph{pure and positional} strategy (also called a \emph{policy}) for player Min is a function $\sigma \colon \Minvertices \to V$ that to each node controlled by Min associates a neighbor, $(v,\sigma(v)) \in E$. Player Min plays according to $\sigma$ if they always move the pawn to $\sigma(v)$ when the pawn lands on $v$. We analogously define a policy $\tau \colon \Maxvertices \to V$ of player Max. Note that if we fix a pair of policies and an initial position of the pawn, then the resulting movement of the pawn is a Markov chain on the graph $(V,E)$. Given an initial position $u \in V$ of the pawn and a pair of policies $(\sigma,\tau)$, the \emph{mean payoff reward} of player Max is the quantity
\[
g(\sigma,\tau)_u \coloneqq \lim_{N \to \infty} \frac{1}{N}\mathbb{E}^u_{\sigma,\tau}(\payoff_{u_0u_1} + \dots + \payoff_{u_{N-1}u_{N}}) \, ,
\]
where $u = u_0, u_1, \dots, u_N$ is the path of the pawn in the first $N$ turns of the game and the expectation is taken with respect to the Markov chain obtained by fixing $(\sigma,\tau)$. The ergodic theorem of Markov chains shows that $g(\sigma,\tau)_u$ is well-defined. The objective of player Max is to maximize the mean payoff reward while the objective of player Min is to minimize this reward. It is known that both players in a stochastic mean payoff game have optimal policies \cite{liggett_lippman}. In other words, there exist policies $(\sigma^*,\tau^*)$ and a vector $\chi \in \R^{\vertices}$ such that for every $u$ and every $(\sigma,\tau)$ we have the inequality
\[
g(\sigma^*,\tau)_u \le \chi_u \le g(\sigma,\tau^*)_u \, .
\]
The reward $\chi = g(\sigma^*,\tau^*)$ is called the \emph{value} of the game. In order to make these games computationally tractable, we suppose that the weights and the probability distributions at nodes in $\Randvertices$ are rational. The nodes of the graph are also called the \emph{states} of the game.

A \emph{simple stochastic game} is a special case of a stochastic mean payoff game. In this game, every node in $\Randvertices$ has out-degree exactly two and the probability distribution at every such node is $(1/2,1/2)$. Furthermore, the graph is equipped with two special terminal states called $\Win$ and $\Lose$. These states are absorbing, i.e., $\Win$ has only one outgoing edge that goes back to $\Win$ and $\Lose$ has only one outgoing edge that goes back to $\Lose$. The weight of the edge going from $\Win$ to $\Win$ is equal to $1$, while the weight of all of the other edges of the graph is equal to $0$. In this case, the mean payoff reward of player Max is simply the probability that the pawn reaches the winning state, $g(\sigma,\tau)_u = \mathbb{P}^{u}_{\sigma,\tau}(\text{Reach}(\Win))$. 

A simple stochastic game is called \emph{stopping} if for every initial position $u \in V$ and every pair of policies $(\sigma,\tau)$, the pawn reaches a terminal state with probability one, $\mathbb{P}^{u}_{\sigma,\tau}(\text{Reach}(\Win,\Lose)) = 1$. In other words, the game is stopping if the Markov chain obtained by fixing any pair of policies has only two recurrent classes, namely $\{\Win\}$ and $\{\Lose\}$.

The computational problem associated with stochastic mean payoff games is the following: given a stochastic mean payoff game, find its value and a pair of optimal policies. In \cite{andersson_miltersen}, Andersson and Miltersen proved that this problem is polynomial-time (Turing) reducible to the problem of finding the value of a stopping simple stochastic game.
By a binary search argument, the problem of computing the value of a stopping simple stochastic game can be further reduced to the decision problem, in which we want to decide whether the value of a given state of such a game is at least~$1/2$. 
The details of this reduction can be found, e.g., in \cite{issac2016jsc} or \cite[Section~7.1]{skomra_phd}.

\begin{theorem}[{\cite[Lemma~48]{issac2016jsc}}] \label{thm:smpg-to-stopping}
The problem of computing the value and a pair of optimal policies of a stochastic mean payoff game is polynomial-time (Turing) reducible to the problem of deciding whether a given state of a stopping simple stochastic game has value at least $1/2$.
\end{theorem}

\section{From stochastic games to 
max-average constraints}
\label{sect:max-avg}
In this section, we focus on the decision problem stated in \cref{thm:smpg-to-stopping} above: given a stopping simple stochastic game we want to decide whether the value of a given state is at least $1/2$. 
This problem can be reduced to max-average constraint satisfaction problems (\emph{max-average CSPs}). 
The ideas for this reduction are not new, but we are not aware of an explicit reference for this fact, and we therefore present a short proof in this section. 
In the next section, we will show that max-average CSPs, and even the larger class of max-plus-average CSPs, then further reduce to the 
feasibility problem of non-Archimedean semidefinite programs. 
The fact that stochastic games can be encoded as non-Archimedean SDPs was proven in \cite{issac2016jsc} and our reduction is very similar.  
While their approach relies on reducing stochastic games to ergodic games, we omit this step and obtain a slightly more straightforward reduction. 

\begin{definition}\label{def:mma}
    An instance of a \emph{max-plus-average CSP} is a conjunction of constraints, each of which has one of the following forms: 
    \begin{enumerate}
        \item 
    $x_0 \leq \max(x_1,\dots,x_k)$ where $k \in \N$, 
    \item $x_0 \leq c+x_1$ where $c$ is a rational constant, 
        \item $x_0 \leq \frac{x_1+x_2}{2}$,
        \item $x_0 = c$ where $c$ is a rational constant. 
    \end{enumerate}
    The computational task is to determine whether such a conjunction has a solution over $\Q \cup \{-\infty\}$ (equivalently, over $\R \cup \{-\infty\} $).
    If the instance does not contain constraints of the form $2.$, we call it a max-average instance.
\end{definition}

Note that max-average CSP instances may also use constraints of the form $x_0 \leq \min(x_1,\dots,x_k)$ where $k \in \N$, $x_0 \leq c$, or $x_0 \geq c$ for some rational constant $c$, because the respective relations are primitively positively definable using max-average constraints.
 
 We now show that the problem of solving stochastic mean payoff games can be reduced to max-average CSPs.

 A reduction to max-plus-average CSPs follows from \cite[Theorem~18]{issac2016jsc} and is based on a Collatz--Wielandt property of order-preserving and additively homogeneous maps. 
 We provide a more direct reduction to the smaller class of CSPs without constraints of the form $x_0 \le c + x_1$ as an intermediate step for our reduction. 
 To do so, we give a proposition that characterizes the value of a stopping simple stochastic game.

 It was already proven in \cite{condon} that this value is a unique fixed point of an operator involving min, max, and averaging operations. By replacing the fixed point equations with inequalities, we obtain a set of certificates which prove that the value is at least $1/2$. This observation is proven, e.g., in \cite{stochasticchapter} and gives inequalities very similar to the inequalities on Shapley operators used in \cite{issac2016jsc}.
 
\begin{proposition}\label{value_ineq}
Consider a stopping simple stochastic game played on a directed graph $\dgraph=([n],E)$. Then, the value of a state $k \in [n]$ is at least $1/2$ if and only if the following system of inequalities (an instance of a max-average constraint satisfaction problem) has a solution $x 
\in \Q_{\geq 0}^n$: 
\begin{equation}\label{eq:gameval}
\begin{aligned}
 x_i &\le \max_{(i,j) \in \edges}\{x_j\} && \text{for every } i \in \Maxvertices, \\
 x_i &\le \min_{(i,j) \in \edges}\{x_j\} && \text{for every }  i \in \Minvertices,  \\
x_i &\le \frac{1}{2}\sum_{(i,j) \in \edges}x_{j}  && \text{for every } i \in \Randvertices,\\
x_{\Lose} &= 0,  \\
x_{\Win} &= 1,\\
x_k &\ge 1/2.
\end{aligned}
\end{equation}
\end{proposition}
\begin{proof}
Let $(Q)$ denote the system of inequalities obtained from \cref{eq:gameval} by removing the inequality $x_k \ge 1/2$. By \cite{condon} or \cite[Theorem~77]{stochasticchapter}, the value of the game satisfies $(Q)$ with inequalities replaced by equalities. Hence, if the value at state $k$ is at least $1/2$, then the value of the game satisfies \cref{eq:gameval}. Moreover, the value of the game is a rational vector by \cite[Lemma~2]{condon} or \cite[Corollary~18]{stochasticchapter}. Conversely, if a point $x \in \Q_{\geq 0}^n$ satisfies $(Q)$, then \cite[Lemma~83]{stochasticchapter} implies that for all $i \in [n]$, the entry $x_i$ is not greater than the value of the game at state $i$. Hence, if we further suppose that $x_k \ge 1/2$, then the value of the game at state $k$ is at least $1/2$.
\end{proof}

\section{From max-plus-average constraints to non-Archimedean SDPs}\label{sec:nonArchSDP}
By ``lifting'' the inequalities 
in Definition~\ref{def:mma}
to Puiseux series, we get the reduction from max-plus-average constraint satisfaction problems to non-Archimedean SDPs. 
A \emph{(formal real) Puiseux series} is a series of the form
\[
\bm{x} = \bm{x}(t) = c_0t^{a/n} + c_1t^{(a-1)/n} + c_2t^{(a-2)/n} + \dots \, ,
\]
where $t$ is a formal parameter, the number $n$ is a positive integer, $a$ is an integer, and the coefficients $c_i$ are real. We further assume that $c_0 \neq 0$ for every Puiseux series except the zero series. The number $a/n$ is called the \emph{valuation} of the series $\bm{x}$ and is denoted by $\val(\bm{x})$, with the convention that $\val(0) = -\infty$.  Puiseux series can be added and multiplied in the natural way. Moreover, they can be ordered by defining $\bm{x} \ge 0$ if $c_0 \ge 0$ and $\bm{x} \ge \bm{y}$ if $\bm{x} - \bm{y} \ge 0$. It is known that Puiseux series form a real closed field, see, e.g., \cite[Section~2.6]{basu_pollack_roy_algorithms} for a detailed discussion.\footnote{In this work, we use a convention in which one interprets $t$ as a ``very large'' parameter. Many authors use the opposite but equivalent convention in which $t$ is a very small parameter. This explains the differences in notation between our work and \cite{basu_pollack_roy_algorithms}.} This field is non-Archimedean, meaning that the order does not satisfy the Archimedean property: the series $\bm{x}(t) = t$ is larger than any natural number, i.e., $t \ge m$ for all $m \in \N$. We denote the field of Puiseux series by $\puiseux$. For any $\bm{x}, \bm{y} \in \puiseux$, the following properties of the valuation function are immediate from the definitions:
\begin{equation}\label{eq:valuation}
\begin{aligned}
\val(\bm{x} + \bm{y}) &\le \max\{\val(\bm{x}), \val(\bm{y})\} \, ,\\
\val(\bm{x}\bm{y}) &= \val(\bm{x}) + \val(\bm{y}) \, ,\\
\bm{x} \ge \bm{y} \ge 0 &\implies \val(\bm{x}) \ge \val(\bm{y}) \, .
\end{aligned}
\end{equation}

\begin{theorem}\label{nonarchSDP}
An instance $\phi$ with variables $x_1,\dots,x_n$ of a max-plus-average constraint satisfaction problem has a solution in $(\Q \cup\{-\infty\})^n$
if and only if the following system of inequalities $S_\phi(t)$, which is a (quadratic) semidefinite program over Puiseux series, has a solution in $\puiseux^n$: 
for every variable $x_i$ of $\phi$ we have a variable $\bm{x}_i$, and the inequalities are 
\begin{equation}\label{eq:puiseuxval}
\begin{aligned}
\bm{x}_{i_0} &\le \bm{x}_{i_1} + \cdots + \bm{x}_{i_k} && \text{ for every constraint $x_{i_0} \leq \max(x_{i_1},\dots,x_{i_k})$ in $\phi$,} \\
\bm{x}_{i_0} &\le t^c \bm{x}_{i_1} &&
\text{ for every constraint $x_{i_0} \leq c+x_{i_1}$,} \\
\bm{x}^2_{i_0} &\le \bm{x}_{i_1}\bm{x}_{i_2} && \text{ for every constraint $x_{i_0} \le \frac{x_{i_1} + x_{i_2}}{2}$  in $\phi$,} \\
\bm{x}_{i_0} & = t^c && \text{ for every constraint $x_{i_0} = c$  in $\phi$,} \\
\bm{x}_{i_0} & \geq 0 && \text{ for every $i_0 \in \{1,\dots,n\}$.}
\end{aligned}
\end{equation}
\end{theorem}

\begin{proof}
Suppose that $\bm{x} = (\bm{x}_1,\dots,\bm{x}_n) \in \puiseux^n$ satisfies \cref{eq:puiseuxval}. Then, $\val(\bm{x})$ belongs to $(\Q \cup \{-\infty\})^n$ and satisfies  $\phi$
by the properties of the valuation function \cref{eq:valuation}. 
Conversely, if $(x_1,\dots,x_n) \in (\Q \cup \{-\infty\})^n$ satisfies $\phi$, then $(t^{x_1},\dots,t^{x_n})$ (with the convention that $t^{-\infty} = 0$) satisfies \cref{eq:puiseuxval}. 

This amounts to the feasibility problem for a semidefinite program: 
the linear (in)equalities can be expressed by diagonal matrices and each  quadratic condition in combination with the non-negativity of all variables is just the positive semidefiniteness of a symmetric 2-by-2 matrix. 
Indeed, the symmetric matrix $\begin{pmatrix} x_1 & x_3 \\ x_3 & x_2 \end{pmatrix}$ is positive semidefinite if and only if $x_1 \ge 0$, $x_2 \ge 0$, and $x_1x_2 \ge x_3^2$. (More generally, a symmetric matrix is positive semidefinite if and only if all of its principal minors are nonnegative.) These are precisely the inequalities that appear in \cref{eq:puiseuxval}, together with some linear inequalities. A conjunction of such inequalities can then be expressed by a single LMI given by block-diagonal matrices 
(a symmetric block-diagonal matrix is positive semidefinite if and only if every block is positive semidefinite).
\end{proof}

\section{From non-Archimedean to real SDPs}
To conclude, we aim to reduce the non-Archimedean semidefinite program from \cref{eq:puiseuxval} to a real semidefinite program. 
We will do this in two steps. First, we replace the formal parameter $t$ by a real number that is sufficiently large (doubly exponential in the size of the instance). In fact, any non-Archimedean semialgebraic feasiblity problem can be reduced to a semialgebraic feasibility problem over the reals by replacing $t$ with a sufficiently large real number. This follows 
from the quantifier elimination in real closed fields coupled with the definition of the order in Puiseux series. By using the effective bounds on quantifier elimination discussed, e.g., in \cite{basu_pollack_roy_algorithms}, we can give a bound on this ``sufficiently large'' number. Second, we will show that this SDP with very large coefficients can be reduced to another SDP that has coefficients of polynomial bit-size. 

Before presenting the proofs, we note that a reduction from non-Archimedean SDPs coming from games to real SDPs with large coefficients was already studied in \cite[Theorem~36]{issac2016jsc}. However, the bounds given in \cite{issac2016jsc} rely on an additional assumption on the underlying game. In particular, they cannot be applied directly to the SDP given in \cref{eq:puiseuxval}. For max-average systems coming from games, one could overcome this restriction by perturbing the underlying game slightly. Nevertheless, the bounds coming from quantifier elimination which we present below do not require any additional assumptions and can be of independent interest. We recall that Tarski's quantifier elimination allows to eliminate quantifiers of first-order $\tau$-formulas, with signature $\tau$ given by $\tau = \{+,\cdot,0,1,\leq\}$.

\begin{theorem}\label{real_parameter}
Let $P_1, \dots, P_m \in \Z[x_0,x_1, \dots, x_n]$ be a collection of polynomials with integer coefficients and let $\diamond \in \{\le,\ge,<,>,=, \neq\}^m$ be a sign pattern. Consider the semialgebraic set over Puiseux series defined by 
\[
\bm{S} \coloneqq \{\bm{x} \in \puiseux^n \colon \forall i \in [m], \, P_i(t,\bm{x}) \diamond_i 0\} \, .
\]
Furthermore, for any real $L > 0$ consider the semialgebraic set over the real numbers defined by 
\[
S(L) \coloneqq \{x \in \R^n \colon \forall i \in [m], \, P_i(L,x) \diamond_i 0\} \, .
\]
Then, $\bm{S}$ is nonempty if and only if $S(L)$ is nonempty for all sufficiently large $L$. More precisely, there exists an absolute constant $C > 0$ with the following property. If 
the maximal degree of any polynomial $P_i$ is bounded by $d \ge 2$ 
and $s$ denotes the maximal bit-size of any coefficient of $P_i$, then either for 
all $L > 2^{sd^{Cn}}$ 
we have $S(L) = \emptyset$,
or for all $L > 2^{sd^{Cn}}$ we have $S(L) \neq \emptyset$. In the former case $\bm{S}$ is empty and in the latter case $\bm{S}$ is nonempty.
\end{theorem}
\begin{proof}
Both sets $S(L)$ and $\bm{S}$
can be defined by the same first-order formula with a free variable $x_0$ that can be substituted by a real number $L$ (to obtain $S(L)$) or a Puiseux series $t$ (to obtain $\bm{S}$).   Let $\tau = \{+,\cdot,0,1,\leq\}$ and consider the first-order $\tau$-formula $\phi(x_0)$ given by
\[
\exists x_1, \dots, \exists x_n \bigland_{i = 1}^{m}(P_i(x_0,x) \diamond_i 0) \, .
\]
By Tarski's quantifier elimination theorem~\cite[Theorem~2.77]{basu_pollack_roy_algorithms}, there exists a quantifier-free formula $\psi(x_0)$ that is equivalent to $\phi(x_0)$ in the theory of real closed fields. The bounds on quantifier elimination given in \cite[Theorem~14.6]{basu_pollack_roy_algorithms} show that $\psi(x_0)$ can be taken to be of the form
\[
\biglor_{i = 1}^{I} \bigland_{j = 1}^{J_i} \left (\biglor_{k = 1}^{N_{ij}} \sign(Q_{ijk}(x_0)) = \delta_{ijk} \right ) \, ,
\]
where $\delta_{ijk} \in \{-1,0,1\}$ and $Q_{ijk} \in \Z[x_0]$ are univariate polynomials with integer coefficients. Furthermore, every $Q_{ijk}$ has degree bounded by $d^{O(n)}$ and every coefficient of every $Q_{ijk}$ has bit-size bounded by $sd^{O(n)}$. Hence, the Cauchy bound \cite[Corollary~10.4]{basu_pollack_roy_algorithms} implies that every real root of $Q_{ijk}$ has absolute value bounded by $2^{sd^{O(n)}}$ (provided that $Q_{ijk}$ is a non-zero polynomial). In particular, the sign of $Q_{ijk}(L)$ is the same for all $L > 2^{sd^{O(n)}}$. Therefore, either $S(L)$ is empty for all such $L$ or $S(L)$ is nonempty for all such $L$. Moreover, $Q_{ijk}(t) \in \puiseux$ is a Puiseux series with finitely many terms. By definition, the sign of this series is the same as the sign of $Q_{ijk}(L) \in \R$.
Hence, $\bm{S}$ is nonempty if and only if $S(L)$ is nonempty.
\end{proof}

In the next proposition and later, 
we use the following notation. Fix an instance $\phi$ of a max-average CSP and let $(c_i)_i$ be the rational constants used in the constraints of $\phi$. For every $i$ we denote $c_i = p_i/q_i$ where $p_i \in \Z$ and $q_i \in \N^*$ and we put $W \coloneqq 2 + \max_i\{|p_i| + q_i\}$.
\begin{proposition}\label{prop:toSDP}
There exists an absolute integer constant $M > 0$ such that 
an instance $\phi$ of a max-plus-average constraint satisfaction problem 
has a solution if and only if 
the real SDP $S_\phi(L)$ obtained from $S_\phi(t)$ (see~\cref{nonarchSDP}) by replacing $t$ with $L > 2^{2^{Mn\log_2 W}}$, is feasible.
\end{proposition}
\begin{proof}
Consider the system of inequalities given in \cref{eq:puiseuxval}. Since $\bm{x}_i \ge 0$ for all $i \in \{1,\dots,n\}$, any equality of the form $\bm{x}_i = t^{c_j} = t^{p_j/q_j}$, for an integer $q_j \geq 1$, can be replaced with an equivalent equality $\bm{x}_i^{q_j} = t^{p_j}$ if $p_j \ge 0$ and $t^{-p_j}\bm{x}_i^{q_j} = 1$ if $p_j < 0$.
Likewise, any inequality of the form $\bm{x}_i \le t^{p_j/q_j} \bm{x}_l$ can be replaced by the equivalent inequality $\bm{x}_i^{q_j} \le t^{p_j}\bm{x}_l^{q_j}$ if $p_j \ge 0$ and $t^{-p_j}\bm{x}_i^{q_j} \le \bm{x}_l^{q_j}$ if $p_j < 0$. 
If we further replace the parameter $t$ by a variable symbol $x_0$, \cref{eq:puiseuxval} becomes a system of polynomial inequalities of degree at most $W$ and such that the coefficients of the polynomials appearing in these inequalities are integers in $\{-1,0,1\}$. Hence, by \cref{real_parameter}, there exists $M > 0$ such that the system \cref{eq:puiseuxval} has a solution over Puiseux series if and only if the system obtained by replacing $t$ with $L > 2^{2^{Mn\log_2 W}}$ has a solution over the real numbers. Therefore, the claim follows from \cref{nonarchSDP}.
\end{proof}

In \cref{prop:toSDP}, we obtain a (real) semidefinite program, whose size, however, is exponential in the size of $\phi$ and therefore does not provide a polynomial-time reduction to the semidefinite program feasiblity problem. 
We now show how to reduce to a semidefinite program of polynomial representation size. This is based on the following result.

\begin{theorem}\label{thm:large-constants}
    For every $n \in {\mathbb Z}$, the singleton set $\{2^{n}\}$ is the projection of a spectrahedron whose representation size is polynomial in the bit-size of $n$. 
\end{theorem}
\begin{proof}
To prove the claim, we extend the constructions of SDPs with solutions of doubly-exponential size given in \cite{KhachiyanPorkolab,Ramana}. For $n = 0$, the set $\{x \in \R \mid x = 1\}$ is the desired spectrahedron. Suppose that $n 
\neq 0$ and let $|n| = \sum_{i = 0}^{k}\bar{b}_{k-i}2^i$ be the binary representation of $|n|$ for $\bar{b}_i 
\in \{0,1\}$. Let $b_i = \sign(n)\bar{b}_i \in \{-1,0,1\}$ for all $i$, so that $n = \sum_{i = 0}^{k}b_{k-i}2^i$.
We first show that the set $\{x \in {\mathbb R} \mid x \geq 2^{n} \}$ 
    is the projection of a spectrahedron $S$ of representation size that is polynomial in the bit-size of $n$. Let $S$ be defined by the linear matrix inequality
    \begin{align*}
    A_0 x_0 + \cdots + A_k x_k - B = 
    \begin{pmatrix}
        x_0 & 1 & \\
        1 & 2^{-b_0} \\
        & & x_1 & x_0 \\
        & & x_0 & 2^{-b_1} \\
        & & & & \ddots \\
        & & & & & x_k & x_{k-1} \\
        & & & & & x_{k-1} & 2^{-b_k} 
    \end{pmatrix} \succeq 0 \, .
    \end{align*}
This LMI expresses that
\[
x_0 \ge 2^{b_0}, \, x_1 \ge 2^{b_1}x_0^2, \, \dots, \, x_k \ge 2^{b_k}x_{k-1}^2 \, .
\]
By induction, for every $\ell \ge 0$ we get $x_{\ell} \ge 2^{b_0 2^{\ell} + b_1 2^{\ell-1} + \dots + b_\ell 2^0}$. Moreover, the vector $(x_0,\dots,x_k)$ with $x_{\ell} = 2^{b_0 2^{\ell} + b_1 2^{\ell-1} + \dots + b_\ell 2^0}$ 
for every $\ell \in \{0,\dots,k\}$ 
satisfies this LMI, so we obtain the desired set by projecting $S$ to the last coordinate. We use strong duality to obtain the singleton set $\{2^{n}\}$ as follows. Consider the SDP 
    \[
    \inf \, (0 \cdot x_0 + \cdots + 0 \cdot x_{k-1} + 1 \cdot x_k)
    \]
    over the spectrahedron $S$ specified above. Then, the dual is of the form $\sup \, \langle B, Y \rangle$ 
    such that $\langle A_i, Y \rangle = c_i$ for every $i \in \{0,\dots,k\}$ and $Y \succeq 0$, where
    $(c_0,\dots,c_k) = (0,\dots,0,1)$. 
    Clearly, the primal is strictly feasible and bounded by $2^{n}$ from below, so we may use strong duality (\cref{thm:strong-dual})
    and obtain that $\sup \, \langle B,Y \rangle = 2^{n}$ and that the supremum is attained for some $Y$. 
    Then the following primitive positive formula over spectrahedral sets 
    with one free variable $x_k$ defines the singleton relation $\{2^{n}\}$:
    \begin{equation}\label{eq:singleton}
    \begin{aligned}
        \exists x_0,\dots,x_{k-1}, y_{1,1},\dots,y_{m,m}: A_0 x_0 + \cdots + A_k x_k - B \succeq 0 & \\
        \wedge \, \langle A_0,Y \rangle = 0 \, \wedge \cdots \wedge \, \langle A_{k-1},Y \rangle = 0 \, \wedge \, \langle A_k,Y \rangle = 1 
        & \wedge Y \succeq 0 \\
        & \wedge \langle B,Y \rangle = x_k 
    \end{aligned}
    \end{equation}
    where $Y$ denotes $(y_{i,j})_{i,j \in \{1,\dots,m\}}$ for $m = 2(k+1)$.
    As we have mentioned before, spectrahedral shadows are closed under primitive positive definability, and the representation size of the resulting LMI is polynomially bounded by the representation sizes for the conjuncts. 
    Indeed, the matrices $A_i$ and $B$ 
    have coefficients in ${0, \pm 1/2, \pm 1, \pm 2}$ and are of size $k \times k$ which is polynomial in the bit-size of $n$. Hence, the system of matrix inequalities \cref{eq:singleton} can be encoded using a number of bits that is polynomial in the bit-size of $n$. This system of inequalities can be then encoded by a single block-diagonal LMI as in the proof of \cref{nonarchSDP}. This increases the number of necessary bits by a polynomial factor.
    We thus obtain an LMI of 
    representation size that is polynomial in $k$. 
\end{proof}

\begin{corollary}\label{cor:MaxAv-SDP}
    There is a polynomial-time reduction from the max-average constraint satisfaction problem to the feasibility problem for semidefinite programs. 
\end{corollary}
\begin{proof}
    We have already specified a first reduction to a (real) SDP in \cref{prop:toSDP}. 
    To obtain a proper polynomial-time reduction, we modify this SDP as follows. 
    Fix a max-average instance $\phi$; let $c_i$, $p_i$, $q_i$, and $W$ be as introduced before \cref{prop:toSDP}. Define $D \coloneqq \prod_i q_i$ and $K \coloneqq 2D W^{Mn} = D2^{1 + Mn \log_2 W}$, where $M$ is the integer constant from \cref{prop:toSDP}. Furthermore, let $L \coloneqq 2^K$.
    By \cref{prop:toSDP}, the instance $\phi$ has a solution if and only if $S_\phi(L)$ 
    has a solution over the real numbers.
    This SDP has constraints of the form $x_i = L^{c_j} = 2^{c_jK}$.
    Since $K$ is an integer divisible by $D$, the number $c_jK$ is an integer for all $j$. Hence, by introducing new variables, we may 
    replace the constraints $x_j = 2^{c_jK}$ by the corresponding LMIs of small representation size from \cref{thm:large-constants}.
    The resulting SDP can be computed in polynomial time. 
    Furthermore, it is feasible if and only if the original SDP is feasible, which proves the statement. 
    \end{proof}

\begin{corollary}
    There is a polynomial-time reduction from stopping simple stochastic games to the feasibility problem for semidefinite programs. 
\end{corollary}
\begin{proof}
    Combine \cref{cor:MaxAv-SDP} with \cref{value_ineq}.
\end{proof}

We now present a strengthening of Corollary~\ref{cor:MaxAv-SDP} which generalises it from max-average to max-plus-average CSPs.

\begin{theorem}
    There is a polynomial-time reduction from the max-plus-average constraint satisfaction problem to the feasibility problem for semidefinite programs.
\end{theorem}
\begin{proof}
    For any integer $k$, the singleton set 
    $\{2^k\}$ can be efficiently represented as a projection of a spectrahedron (Theorem~\ref{thm:large-constants}). Suppose that this spectrahedron is given by
    \[
    Q^{(0)} + x_1 Q^{(1)} + \dots + x_nQ^{(n)} \succeq 0,
    \]
    and that the projection on the last coordinate
    $x_n$ is the singleton set $\{2^k\}$. 
    Consider now the following spectrahedron, with one more variable $x_0$:
    \[
    x_{0}Q^{(0)} + x_1 Q^{(1)} + \dots + x_nQ^{(n)} \succeq 0 \wedge x_{0} \ge 0.
    \]
We claim that the projection of this spectrahedron to the coordinates $(x_0,x_n)$ is precisely the closed half-line defined by \[ x_n = 2^k x_0 \wedge x_0 \ge 0. \] 
The proof has two steps. First, consider the projection to $(x_0,x_n)$ of the subset defined by  
\[
x_{0}Q^{(0)} + x_1 Q^{(1)} + \dots + x_nQ^{(n)} \succeq 0 \wedge x_{0} > 0.
\]
This projection is the open half-line
defined by $x_n = 2^k x_0 \wedge  x_0 > 0$.
Indeed, take a point $(x_0,x_n)$ such that $x_0 > 0$. Moreover, suppose that we can find $(x_1,\dots,x_{n-1})$ such that the matrix $x_{0}Q^{(0)} + x_1 Q^{(1)} + \dots + x_nQ^{(n)}$ is positive semidefinite.
Then the matrix
\[ \frac{1}{x_0} ( x_{0}Q^{(0)} + x_1 Q^{(1)} + \dots + x_nQ^{(n)} ) = Q^{(0)} + \frac{x_1}{x_0} Q^{(1)} + \dots + \frac{x_n}{x_0}Q^{(n)} \]
is also positive semidefinite (here we use the assumption that $x_0 > 0$). Therefore, we get $\frac{x_n}{x_0} = 2^k$ by the definition of the matrices $Q^{(i)}$. In particular, the projection is included in the half-line. 

Conversely, if $(x_0,x_n)$ are on the open half-line, then $\frac{x_n}{x_0} = 2^k$ and the definition of the matrices $Q^{(i)}$ implies that there exist $(\tilde{x}_1, \dots, \tilde{x}_n)$ such that $Q^{(0)} + \tilde{x}_1 Q^{(1)} + \dots + \tilde{x}_{n-1}Q^{(n-1)} + \frac{x_n}{x_0}Q^{(n)} \succeq 0$. Hence $x_0Q^{(0)} + x_0\tilde{x}_1 Q^{(1)} + \dots + x_0\tilde{x}_{n-1}Q^{(n-1)} + x_nQ^{(n)} \succeq 0$ and $(x_0,x_n)$ belong to the projection.

Second, since the projection of a spectrahedron is convex, the projection of \[x_{0}Q^{(0)} + x_1 Q^{(1)} + \dots + x_nQ^{(n)} \succeq 0 \wedge x_{0} \ge 0\]
is either the open half-line or the closed half-line (we cannot add to the projection any other point $(x_0,x_n)$ with $x_0 = 0$ without destroying the convexity of the projection). Since $(0,0)$ is in the projection, we get that the projection is exactly the closed half-line as claimed.

We have thus expressed the constraints of the form $x_0 = 2^k x_1$ as small SDPs. Since the constraint $x_0 \le 2^k x_1$ can be simulated by two constraints of the form $x_0 \le y$, $y = 2^k x_1$, we can also express constraints of the form $x_0 \le 2^k x_1$ as small SDPs.
By using the same reasoning as in the proof of Corollary~\ref{cor:MaxAv-SDP}, 
this allows us to simulate the non-Archimedean constraints of the form $x_0 \le t^c x_1$, and thus max-plus-average CSPs by \cref{nonarchSDP}. 
\end{proof}

\bibliography{phd_bibliography,global}

\end{document}